\documentclass[12pt, reqno]{amsart}

\newtheorem{theorem}{Theorem}[section]
\newtheorem{lemma}[theorem]{Lemma}
\newtheorem{proposition}[theorem]{Proposition}
\newtheorem{corollary}[theorem]{Corollary}   
\newtheorem{definition}[theorem]{Definition}
\newtheorem{example}[theorem]{Example}

\newtheorem{remark}[theorem]{Remark}
\newtheorem{conjecture}[theorem]{Conjecture}
\newtheorem{question}[theorem]{Question}

\numberwithin{equation}{section}

\usepackage{times}
\usepackage{enumerate}
\usepackage{mathrsfs}
\usepackage{tfrupee}
\usepackage{tikz}
\usetikzlibrary{chains,fit}
\usetikzlibrary{shapes,snakes}
\usetikzlibrary{graphs}
\usepackage{graphicx,adjustbox}
\usepackage{hyperref}
\usepackage{amsmath,amssymb}
\usepackage{amscd}
\usepackage{graphicx}
\usepackage[all]{xy}

\usepackage{float}
\usepackage{caption}
\usepackage{subcaption}

\usepackage{cleveref}

\usepackage{geometry}
\begin{document}

\title{The $\mathrm{v}$-Number of Binomial Edge Ideals}
\author{
Siddhi Balu Ambhore, Kamalesh Saha, \and Indranath Sengupta
}
\date{}

\address{\small \rm  Discipline of Mathematics, IIT Gandhinagar, Palaj, Gandhinagar, 
Gujarat 382055, INDIA.}
\email{siddhi.ambhore@iitgn.ac.in}

\address{\small \rm  Discipline of Mathematics, IIT Gandhinagar, Palaj, Gandhinagar, 
Gujarat 382055, INDIA.}
\email{kamalesh.saha@iitgn.ac.in}

\address{\small \rm  Discipline of Mathematics, IIT Gandhinagar, Palaj, Gandhinagar, 
Gujarat 382055, INDIA.}
\email{indranathsg@iitgn.ac.in}
\thanks{The third author is the corresponding author}

\date{}



\allowdisplaybreaks

\begin{abstract}
The invariant $\mathrm{v}$-number was introduced very recently in the study of Reed-Muller-type codes. Jaramillo and Villarreal (J Combin. Theory Ser. A 177:105310, 2021) initiated the study of the $\mathrm{v}$-number of edge ideals. Inspired by their work, we take the initiation to study the $\mathrm{v}$-number of binomial edge ideals in this paper. We discuss some properties and bounds of the $\mathrm{v}$-number of binomial edge ideals. We explicitly find the $\mathrm{v}$-number of binomial edge ideals locally at the associated prime corresponding to the cutset $\emptyset$. We show that the $\mathrm{v}$-number of Knutson binomial edge ideals is less than or equal to the $\mathrm{v}$-number of their initial ideals. Also, we classify all binomial edge ideals whose $\mathrm{v}$-number is $1$. Moreover, we try to relate the $\mathrm{v}$-number with the Castelnuvo-Mumford regularity of binomial edge ideals and give a conjecture in this direction.
\end{abstract}

\maketitle

\section{Introduction}

Let $R=K[x_{1}\,\ldots, x_{n}]=\bigoplus_{d=0}^{\infty}R_{d}$ 
denote the polynomial ring in $n$ variables over a field $K$ with the standard grading. For a graded ideal $I$ of $R$, the set of \textit{associated prime} ideals of $I$, denoted by $\mathrm{Ass}(I)$ or $\mathrm{Ass}(R/I)$, is the collection of prime ideals of $R$ of the form $(I:f)$ for some $f\in R_{d}$. In 2020, Cooper et al. introduced a new invariant, called v-number, for graded ideals of $R$ during the study of Reed-Muller-type codes \cite{cstvv}.

\begin{definition}[{\cite[Definition 4.1]{cstvv}}]\label{v1.1}{\rm
Let $I$ be a proper graded ideal of $R$. The $\mathrm{v}$-\textit{number} of $I$, denoted 
by $\mathrm{v}(I)$, is defined by
$$ \mathrm{v}(I):=
 \mathrm{min}\{d\geq 0 \mid \exists\, f\in R_{d}\,\,\text{and}\,\, \mathfrak{p}\in \mathrm{Ass}(I)\,\, \text{with}\,\, (I:f)=\mathfrak{p} \}.$$
For each $\mathfrak{p}\in \mathrm{Ass}(I)$, we can locally define v-number as
$$\mathrm{v}_{\mathfrak{p}}(I):=\mathrm{min}\{d\geq 0 \mid \exists\, f\in R_{d}\,\, \text{with}\,\, (I:f)=\mathfrak{p} \}.$$
Then $\mathrm{v}(I)=\mathrm{min}\{\mathrm{v}_{\mathfrak{p}}(I)\mid \mathfrak{p}\in\mathrm{Ass}(I)\}$.
}
\end{definition}

This invariant of $I$ helps us understand the asymptotic behaviour of the minimum distance function $\delta_{I}$ of projective Reed-Muller-type codes (see \cite{cstvv}). So far, very little is known about the $\mathrm{v}$-number, and \cite{civan23}, \cite{grv21}, \cite{v}, \cite{v-mon} are the only papers written in this direction. The first paper entirely devoted to the $\mathrm{v}$-number was written by Jaramillo and Villarreal \cite{v}, where they studied the $\mathrm{v}$-number of edge ideals. Motivated by their work in \cite{v}, we take the initiation to study the $\mathrm{v}$-number of binomial edge ideals.

\begin{definition}{\rm
Let $G$ be a simple graph on the vertex set $V(G)=[n]=\{1,\ldots,n\}$ with the edge set $E(G)$. Consider the polynomial ring $S=K[x_{1},\ldots,x_{n},y_{1},\ldots,y_{n}]$ over a field $K$. Then the \textit{binomial edge ideal} of $G$, denoted by $J_{G}$, is an ideal of $S$ defined as
$$J_{G}=\big<f_{ij} = x_{i}y_{j}-x_{j}y_{i}\mid \{i,j\}\in E(G)\,\, \text{with}\,\, i<j\big>.$$
}
\end{definition}

The study of binomial edge ideals was started in 2010 independently through the articles \cite{hhhrkara} and \cite{ohtani}. Since then, this has become an intensive research topic in combinatorial commutative algebra. A binomial edge ideal has a natural determinantal structure in the sense that it can be seen as an ideal generated by a set of $2\times 2$-minors of a $2 \times n$ matrix $X$ of indeterminates. One of the primary motivations behind studying these ideals is their connection to algebraic statistics, particularly their appearance in the study of conditional independence statements \cite[Section 4]{hhhrkara}. Moreover, it is proved in \cite{cdg18} that binomial edge ideals belong to the class of \textit{Cartwright-Sturmfels} ideals, which was introduced in \cite{cdg20} inspired by the work of Cartwright and Sturmfels \cite{cs10} and has many nice properties.
\medskip

Generally, people study algebraic properties and invariants of binomial edge ideals by investigating the underlying graphs' combinatorics. So far, lots of studies have been done on binomial edge ideals in several directions (see the survey paper \cite{surveybinom} and references therein). In this paper, we give a new direction in the study of binomial edge ideals by investigating their $\mathrm{v}$-number. We discuss some properties of the $\mathrm{v}$-number of binomial edge ideals and their initial ideals. There are many papers (\cite{ez15}, \cite{ks16}, \cite{jnr19}, \cite{arvind21}, \cite{krs21}) on the upper bound of (\textit{Castelnuovo-Mumford}) regularity of binomial edge ideals, but the general lower bound of the regularity of binomial edge ideals is only given by Matsuda and Murai \cite{mm13}. We try to establish a new lower bound on the regularity of binomial edge ideals using the $\mathrm{v}$-number and give a conjecture on the relation between $\mathrm{v}$-number and regularity of binomial edge ideals. The paper is organized in the following manner: In Section \ref{preli}, we discuss the necessary prerequisites. In Section \ref{v-binomprop}, we study some properties of the $\mathrm{v}$-number of binomial edge ideals. For a vertex $v$ of a graph $G$, we denote by $G_{v}$, the following graph:
$$ V(G_{v})=V(G)\,\, \text{and}\,\, E(G_{v})=E(G)\cup \{\{i,j\}\mid i,j\in \mathcal{N}_{G}(v), i\neq j\}.$$

\noindent Corresponding to a vertex $v$ of a graph $G$, we have the following exact sequence (see \cite[Proof of Theorem 1.1]{ehh_cmbin}):
$$ 0\longrightarrow S/J_{G}\longrightarrow S/J_{G_{v}}\oplus S/\big<J_{G\setminus \{v\}},x_{v},y_{v}\big> \longrightarrow S/\big<J_{G_{v}\setminus\{v\}}, x_{v},y_{v}\big>\longrightarrow 0.$$

\noindent The graphs $G\setminus \{v\}, G_{v}, G_{v}\setminus\{v\}$ play an important role in the study of binomial edge ideals. These graphs and the above exact sequence help inductively to study the invariants (like regularity, depth, etc.) and properties (like unmixed, Cohen-Macaulay, etc.) of $J_{G}$. In Proposition \ref{vproperty}, we show how $\mathrm{v}(J_{G})$ is related to $\mathrm{v}(J_{G_{v}})$ and $\mathrm{v}(J_{G\setminus\{v\}})$. Then we define completion set of $G$ (Definition \ref{completionset}) with minimum and maximum completion number of $G$ (denoted by $\mathrm{min\mbox{-}comp}(G)$ and $\mathrm{max\mbox{-}comp}(G)$, respectively) to find some bounds on $\mathrm{v}(J_{G})$. We explicitly find $\mathrm{v}_{\emptyset}(G)$, the $\mathrm{v}$-number of $J_{G}$ locally at $P_{\emptyset}(G)$, and get a combinatorial upper bound of $\mathrm{v}(J_{G})$ in the following theorem:
\medskip

\noindent \textbf{Theorem \ref{vphi}.} \textit{Let $G$ be a simple graph. Then $\mathrm{v}_{\emptyset}(J_{G})=\mathrm{min\mbox{-}comp}(G)$. In particular, we have $\mathrm{v}(J_{G})\leq \mathrm{min\mbox{-}comp}(G)$.
}
\medskip

\noindent As a corollary of Theorem \ref{vphi}, we get $\gamma(G)\leq \mathrm{v}_{\emptyset}(G)$ in Corollary \ref{cordomin}, where $G$ is a connected non-complete graph and $\gamma(G)$ denotes the domination number of $G$. In Theorem \ref{v-add}, we prove the additivity of $\mathrm{v}$-number for some radical ideals, and as an application of Theorem \ref{v-add}, we get the additivity of $\mathrm{v}$-number of binomial edge ideals as follows:
\medskip

\noindent \textbf{Corollary \ref{v-addbinom}.}\textit{
Let $G=G_{1}\sqcup G_{2}$ be a graph. Then $\mathrm{v}(J_{G})=\mathrm{v}(J_{G_{1}})+\mathrm{v}(J_{G_{2}})$.
}
\medskip

Next, we try to establish a relation between the $\mathrm{v}$-number of binomial edge ideals and their initial ideals. We show that under some circumstances, $\mathrm{v}(J_{G})\leq \mathrm{v}(\mathrm{in}_{<}(J_{G}))$ in Theorem \ref{thmvin} and as an application we get in Corollary \ref{corweakly} that $\mathrm{v}(J_{G})\leq \mathrm{v}(\mathrm{in}_{<}(J_{G}))$ for Knutson binomial edge ideals. In Example \ref{exmvin}, we show that the $\mathrm{v}$-number of initial ideals of binomial edge ideals depends on the labelling of vertices. Finally, we end up this section by classifying all binomial edge ideals with $\mathrm{v}$-number $1$ as follows:
\medskip

\noindent \textbf{Theorem \ref{thmcone}.} \textit{Let $G$ be a simple connected graph. Then $\mathrm{v}(J_{G})=1$ if and only if $G=\mathrm{cone}(v,H)$ for some non-complete graph $H$.
}
\medskip

Section \ref{secvreg} of this paper is devoted to study the relation between regularity and $\mathrm{v}$-number of binomial edge ideals. In \cite[Corollary 2.7]{cv20}, Conca and Varbaro showed that for a graded ideal $I$ in a polynomial ring $R$ with a square-free initial ideal $\mathrm{in}_{<}(I)$ for some term order $<$, we have $\mathrm{reg}(R/I)=\mathrm{reg}(R/\mathrm{in}_{<}(I))$. Using this fact and looking at the initial ideal, we try to give a relation between the $\mathrm{v}$-number and regularity of binomial edge ideals. One of the main results of this section is the following:
\medskip

\noindent \textbf{Theorem \ref{vregchordal}.} \textit{Let $G$ be a chordal graph. Then $\mathrm{max\mbox{-}comp}(G)\leq \mathrm{reg}(S/J_{G})$. In particular, we have $\mathrm{v}(J_{G})\leq \mathrm{v}_{\emptyset}(J_{G})\leq \mathrm{max\mbox{-}comp}(G)\leq \mathrm{reg}(S/J_{G})$.
}
\medskip 

\noindent Also, we show that $\mathrm{v}_{\emptyset}(J_{G})\leq \mathrm{reg}(S/J_{G})$ in Theorem \ref{thmwhisker} for some classes of graphs, including whisker graphs (see Corollary \ref{corwhisker}). In Example \ref{exmwhisker}, we show that $\mathrm{v}_{\emptyset}(J_{G})$ could be a better lower bound for $\mathrm{reg}(S/J_{G})$ than the lower bound given by Matsuda and Murai in \cite{mm13}. Also, we show in Example \ref{v=reg} that this lower bound is tight by providing a graph $G$, which satisfy $\mathrm{v}(J_{G})=\mathrm{v}_{\emptyset}(J_{G})=\mathrm{reg}(S/J_{G})$. At the end, we show that for a given $n\in \mathbb{N}$, there exists a graph $G$ satisfying $\mathrm{reg}(S/J_G) - \mathrm{v}(J_G) = n$ (see Theorem \ref{reg-v}), i.e. the difference between $\mathrm{v}$-number and regularity of $J_{G}$ can be arbitrarily large. 
\medskip

In the last section (Section \ref{secprob}), we put some open problems to give a future direction on the study of $\mathrm{v}$-number of binomial edge ideals. Also, due to Theorem \ref{vregchordal}, Theorem \ref{thmwhisker}, and some evidence from our computations, we conjecture the following:
\medskip

\noindent \textbf{Conjecture \ref{conjvnum}.}\textit{
Let $G$ be a simple graph. Then $\mathrm{v}_{\emptyset}(J_{G})\leq \mathrm{reg}(S/J_{G})$. In particular, we have $\mathrm{v}(J_{G})\leq \mathrm{reg}(S/J_{G})$.
}
\medskip

The authors saw the recent preprint \cite{v-binom-js23} on arXiv, where there are some results identical to our results, however, the approaches are different and done independently.

\section{Preliminaries}\label{preli}

Let $R=K[x_{1},\ldots, x_{n}]$ be a polynomial ring over a field $K$, with the standard gradation. An Ideal $I$ of $R$ is said to be a \textit{monomial ideal} if $I$ is generated by a set of monomials and the unique minimal generating set of $I$ is denoted by $G(I)$. A monomial ideal $I$ is said to be \textit{square-free} if $G(I)$ consists of only square-free monomials. It is a well-known fact that square-free monomial ideals are radical ideals and their associated prime ideals are generated by a set of variables. Let $<$ be a monomial order on $R$. For a graded ideal $I\subseteq R$, we denote the \textit{initial ideal} of $I$ with respect to $<$ by $\mathrm{in}_{<}(I)$. For a monomial $m\in R$, the \textit{support} of $m$, denoted by $\mathrm{supp}(m)$, is defined as $\mathrm{supp}(m):=\{x_{i}\mid x_{i}\,\,\text{divides}\,\, m\}$. In this article, every ideal is assumed to be graded.
\medskip

In this paper, we assume all graphs are simple and whenever applicable connected also. For $T\subseteq V(G)$, we write $G\setminus T$ to denote the induced subgraph of $G$ on the vertex set $V(G)\setminus T$. Again by $G[T]$, we mean the induced subgraph of $G$ on the vertex set $T$. For a vertex $v\in V(G)$, we say $\mathcal{N}_{G}(v)=\{u\in V(G)\mid \{u,v\}\in E(G)\}$ the \textit{neighbour set} of $v$ in $G$. We write $\mathcal{N}_{G}[v]:=\mathcal{N}_{G}(v)\cup \{v\}$. A \textit{path} from $u$ to $v$ of length $n$ in $G$ is a sequence of vertices $u=v_{0},\ldots,v_{n}=v\in V(G)$, such that $\{v_{i-1},v_{i}\}\in E(G)$ for each $1\leq i\leq n$, and $v_{i}\neq v_{j}$ if $i\neq j$. 

\begin{definition}{\rm
A \textit{path} graph on $n$ vertices, denoted by $P_{n}$, is a graph whose vertex set can be ordered as $v_{1},\ldots, v_{n}$ such that $E(P_{n})=\{\{v_{i},v_{i+1}\}\mid 1\leq i\leq n-1\}$. The \textit{length} of $P_{n}$ is the number of edges in $P_{n}$, which is $n-1$. An \textit{induced path} of a graph $G$ is an induced subgraph of $G$, which is a path graph. We denote by $\ell(G)$ the maximum length of an induced path in $G$.
}
\end{definition}

\begin{definition}{\rm
A \textit{cycle} of length $n$, denoted by $C_{n}$, is a graph with $n$ vertices $v_{1},\ldots, v_{n}$ such that $E(G)=\{\{v_{i},v_{i+1}\}\mid 1\leq i\leq n-1\}\cup \{\{v_{1},v_{n}\}\}$. A graph is said to be \textit{chordal} if it has no induced cycle $C_{n}$ for $n\geq 4$.
}
\end{definition}

\begin{definition}{\rm
A graph is said to be \textit{complete} if there is an edge between every pair of vertices. We denote a complete graph on $n$ vertices by $K_{n}$. 
}
\end{definition}

\begin{remark}\label{remGv}{\rm
A vertex $v\in V(G)$ is said to be a \textit{free} vertex of $G$ if $\mathcal{N}_{G}(v)$ is a complete graph. It follows from \cite[Proposition 2.1]{raufrin}, 
that, $v$ is a free vertex of $G$ if and only if $v\not\in T$ for all $T\in\mathscr{C}(G)$. Also, from \cite[Proof of Theorem 1.1]{ehh_cmbin} it is observed that $T\in\mathscr{C}(G_{v})$ if and only if $v\not\in T$ and $T\in\mathscr{C}(G)$.
}
\end{remark}

\begin{definition}{\rm
Let $G$ be a graph with $V(G)=[n]$. A path $\pi: i=i_{0},i_{1},\ldots,i_{r}=j$ from $i$ to $j$ with $i<j$ in $G$ is said to be an \textit{admissible path} if the following hold:
\begin{enumerate}
\item $i_{k}\neq i_{l}$ for $k\neq l$;

\item For each $k\in\{1,\ldots,r-1\}$, we have either $i_{k}<i$ or $i_{k}>j$;

\item The induced subgraph of $G$ on the vertex set $\{i_{0},\ldots,i_{r}\}$ has no induced cycle.
\end{enumerate}
}
\end{definition}

\begin{remark}\label{remgrob}{\rm
Corresponding to an admissible path $\pi: i=i_{0},i_{1},\ldots,i_{r}=j$ from $i$ to $j$ with $i<j$ in $G$, we associate the monomial
$$ u_{\pi}=\bigg(\prod_{i_{k}>j} x_{i_{k}}\bigg)\bigg(\prod_{i_{l}<i} y_{i_{l}}\bigg).$$
Then $\mathcal{G}=\{u_{\pi}f_{ij}\mid \pi\,\, \text{is an admissible path from}\,\, i\,\, \text{to}\,\, j\,\, \text{with}\,\, i<j\}$ is a reduced Gr\"{o}bner basis of $J_{G}$ with respect to $<$ by \cite[Theorem 2.1]{hhhrkara}. Therefore, we have 
$$G(\mathrm{in}_{<}(J_{G}))=\{u_{\pi}x_{i}y_{j}\mid \pi\,\, \text{is an admissible path from}\,\, i\,\, \text{to}\,\, j\,\, \text{with}\,\, i<j\}.$$
}
\end{remark}

\subsection*{Primary decomposition of binomial edge ideals} A vertex $v\in V(G)$ is said to be a \textit{cut vertex} of $G$, 
if removal of $v$ from $G$ increases the number of connected components. Let $G$ be a graph on the vertex set $V(G)=[n]$. A set $T\subseteq [n]$ is said to be a \textit{cutset} of $G$ if each $t\in T$ is a cut vertex of $G\setminus (T\setminus\{t\})$. We denote by $\mathscr{C}(G)$ the set of all cutsets of $G$. For $T\subseteq [n]$, we denote the number of connected components of the graph $G\setminus T$ by $c_{G}(T)$ (or sometimes by $c(T)$ if the graph is clearly understood from the context). Let $G_{1},\ldots,G_{c(T)}$ be the connected components of $G\setminus T$. For each $G_{i}$, we denote by $\tilde{G_{i}}$, the 
complete graph on the vertex set $V(G_{i})$. We set
$$ P_{T}(G)=\left\langle \bigcup_{i\in T}\lbrace x_{i},y_{i}\rbrace, J_{\tilde{G_{1}}},\ldots,J_{\tilde{G}_{c(T)}}\right\rangle.$$
Then $P_{T}(G)$ is a prime ideal. By \cite[Corollary 2.2]{hhhrkara}, $J_{G}$ is a radical ideal and from \cite[Corollary 3.9]{hhhrkara}, the minimal primary decomposition of $J_{G}$ is 
$$J_{G}=\bigcap_{T\in\mathscr{C}(G)} P_{T}(G).$$

\noindent \textbf{Note:} Instead of writing $\mathrm{v}_{P_{T}(G)}(J_{G})$, the $\mathrm{v}$-number of $J_{G}$ locally at an associated prime $P_{T}(G)$, we will denote it by $\mathrm{v}_{T}(J_G)$.

\begin{remark}{\rm
For a prime ideal $\mathfrak{p}$ in $R$, we have $(\mathfrak{p}:1)=\mathfrak{p}$ and hence, $\mathrm{v}(\mathfrak{p})=0$. Note that for a graph $G$, $P_{\emptyset}(G)$ is a disjoint union of binomial edge ideals of complete graphs. Hence, we have $\mathrm{v}(J_{K_{n}})=0$ and if $G$ is a disjoint union of complete graphs, then also $\mathrm{v}(J_{G})=0$.
}
\end{remark}

\begin{remark}{\rm
Let $I$ be a radical ideal with $I=\mathfrak{p}_{1}\cap\cdots\cap \mathfrak{p}_{k}$ as a primary decomposition. Then we have by \cite[Exercise 1.12, Lemma 4.4]{am} that $(I:f)=\mathfrak{p}_{i}$ if and only if $f\not\in \mathfrak{p}_{i}$ and $f\in \mathfrak{p}_{j}$ for all $j\neq i$. We will use this fact frequently in our proofs.
}
\end{remark}
\medskip

\section{Properties of the $\mathrm{v}$-Number of Binomial Edge Ideals}\label{v-binomprop}

In this section, we study some properties of the $\mathrm{v}$-number of binomial edge ideals. We explicitly find $\mathrm{v}_{\emptyset}(J_{G})$ and give a combinatorial upper bound of $\mathrm{v}(J_{G})$. We try to establish the relation between $\mathrm{v}(J_{G})$ and $\mathrm{v}(\mathrm{in}_{<}(J_{G}))$. Finally, we classify all binomial edge ideals with $\mathrm{v}$-number $1$.

\begin{proposition}\label{complete}
Let $G$ be a simple graph on the vertex set $[n]$. Then for any vertex $i \in [n]$, we have $(J_G : x_i) = J_{G_i}$ and $(J_G:y_i)=J_{G_i}$.
\end{proposition}

\begin{proof}
We first prove $J_{G_i} \subseteq (J_G : x_i)$. Note that $J_{G_i} = J_G + \big< f_{pq} \mid p,q \in N_G(i), p <q, \{p,q\} \notin E(G) \big>$. Since $J_{G}\subseteq (J_{G}:x_{i})$, it is enough to show that $f_{pq}\in (J_{G}:x_{i})$ for $p,q\in \mathcal{N}_{G}(i)$ with $p<q$ and $\{p,q\}\not\in E(G)$. Now we can write 
\begin{align*}
x_{i}f_{pq} &= x_i(x_py_q - x_qy_p)\\
& = x_ix_py_q -x_px_qy_i +x_px_qy_i - x_ix_qy_p \\
& = x_pf_{iq} + x_qf_{pi}.
\end{align*}
Since $ \{p,i\}, \{i,q\} \in E(G)$, we get $x_{i}f_{pq} \in J_G$. Hence,  $J_{G_i} \subseteq (J_G : x_i)$.\par 

Now, We will show $(J_G:x_i) \subseteq J_{G_i}$. Let $f \in (J_G:x_i)$. This implies $fx_i \in J_G$. We know that $J_G \subseteq J_{G_i} \subseteq P_T(G_i)$, for all $T \in \mathscr{C}(G_i)$. Thus, we get $fx_i \in P_T(G_i)$, for all $T \in \mathscr{C}(G_i)$. By Remark \ref{remGv}, $\mathscr{C}(G_i) = \{T \in \mathscr{C}(G) \mid i \notin T\}$. Thus, $x_i \notin P_T(G_i)$ for all $T \in \mathscr{C}(G_i)$. Therefore, we get $f \in P_T(G_i)$ for all $T \in \mathscr{C}(G_i)$, which implies $f \in \bigcap_{T \in \mathscr{C}(G_i )} P_T(G_i) = J_{G_i}$. Hence, $(J_G:x_i) \subseteq J_{G_i}$.\par 

Similarly, one can show that $(J_G : y_i) = J_{G_i}$. The proof follows same as the above, where for the first part of the proof, we get $f_{pq}y_i = y_q(f_{pi}) + y_p(f_{iq})$. 
\end{proof}

\begin{proposition}\label{Gij}
Let $G$ be a simple graph on vertex set $[n]$. Then for $i,j \in [n]$ we get $(G_i)_{j} = (G_{j})_{i}$.
\end{proposition}

\begin{proof}
It is enough to prove that $J_{(G_{i})_{j}} = J_{(G_{j})_{i}}$. By Proposition \ref{complete}, we can write $J_{(G_{i})_{j}} = (J_{G_{i}}:x_{j}) = ((J_G:x_{i}):x_{j}) = (J_G:x_{i}x_{j}) = (J_G:x_{j}x_{i}) = ((J_G:x_{j}):x_{i}) = (J_{G_{j}}:x_{i}) = J_{(G_{j})_{i}}$.
\end{proof}

\begin{proposition}\label{vproperty}
Let $G$ be a simple graph. Then the following hold:
\begin{enumerate}
\item[$\mathrm{(a)}$] For any $v\in V(G)$, we have $\mathrm{v}(J_{G})\leq \mathrm{v}(J_{G_{v}})+1$.
\item[$\mathrm{(b)}$] For a vertex $v$ of $G$, if there exists a cutset $T$ of $G$ such that $v\not\in T$ and $\mathrm{v}(J_{G})=\mathrm{v}_{T}(J_{G})$, then $\mathrm{v}(J_{G_{v}})\leq \mathrm{v}(J_{G})$.
\item[$\mathrm{(c)}$] For a vertex $v$ of $G$, if there exists a cutset $T$ of $G$ such that $v\in T$ and $\mathrm{v}(J_{G})=\mathrm{v}_{T}(J_{G})$, then $\mathrm{v}(J_{G\setminus \{v\}})\leq \mathrm{v}(J_{G})$.
\end{enumerate}
\end{proposition}

\begin{proof}
$\mathrm{(a)}$: We know $\mathscr{C}(G_{v})=\{T\in \mathscr{C}(G)\mid v\not\in T\}$ by Remark \ref{remGv}. Let $f$ be a homogeneous polynomial and $T\in \mathscr{C}(G_{v})$ such that $(J_{G_{v}}:f)=P_{T}(G_{v})$ and $\mathrm{deg}(f)=\mathrm{v}(J_{G_{v}})$. Note that $T\in \mathscr{C}(G)$ and $P_{T}(G)=P_{T}(G_{v})$ as $v\not\in T$. Then $(J_{G_{v}}:f)=P_{T}(G_{v})$ implies $(J_{G}:x_{v}f)=P_{T}(G)$ by Proposition \ref{complete}. Hence, $\mathrm{v}(J_{G})\leq \mathrm{v}(J_{G_{v}})+1$.\medskip

$\mathrm{(b)}$: Let $f$ be a homogeneous polynomial such that $(J_{G}: f)=P_{T}(G)$ and $\mathrm{v}(J_{G})=\mathrm{deg}(f)$. Since $v\not \in T$, we have $T\in \mathscr{C}(G_{v})$ by Remark \ref{remGv}. Now $(J_{G}: f)=P_{T}(G)$ implies $f\in P_{T'}(G)$ for all $T'\in\mathscr{C}(G)$ with $T'\neq T$ and $f\not\in P_{T}(G)$. Note that for every $T'\in \mathscr{C}(G_{v})$, we have $P_{T'}(G_{v})=P_{T'}(G)$. Therefore, $f\in P_{T'}(G_{v})$ for all $T'\in \mathscr{C}(G_{v})$ with $T'\neq T$ and $f\not\in P_{T}(G_{v})$. Hence, $(J_{G_{v}}:f)=P_{T}(G_{v})$ and so, $\mathrm{v}(J_{G_{v}})\leq \mathrm{v}(J_{G})$.
\medskip

$\mathrm{(c)}$: $T$ is a cutset of $G$ implies every $t\in T$ is a cut vertex of $G\setminus (T\setminus\{t\})$. Then every $t\in T\setminus\{v\}$ is a cut vertex of $G\setminus (T\setminus\{t\})=(G\setminus\{v\})\setminus ((T\setminus\{v\})\setminus\{t\})$, which gives $T\setminus\{v\}\in \mathscr{C}(G\setminus\{v\})$. Since $v\in T$, we have $P_{T}(G)=\big<x_{v},y_{v}\big>+P_{T\setminus\{v\}}(G\setminus\{v\})$. Since $x_{v},y_{v}\in P_{T}(G)$, we can choose homogeneous polynomial $f\in K[\{x_{i},y_{i}\mid i\in V(G\setminus\{v\})\}]$ such that $(J_{G}:f)=P_{T}(G)$ and $\mathrm{v}(J_{G})=\mathrm{deg}(f)$. Now $fP_{T}(G)\subseteq J_{G}$ implies $fP_{T\setminus\{v\}}(G\setminus\{v\})\subseteq J_{G\setminus\{v\}}$ as $f\in K[\{x_{i},y_{i}\mid i\in V(G\setminus\{v\})\}]$. Thus, $P_{T\setminus\{v\}}(G\setminus\{v\})\subseteq (J_{G\setminus\{v\}}:f)$. Let $g\in (J_{G\setminus\{v\}}:f)$. Then $fg\in J_{G\setminus\{v\}}\subseteq P_{T\setminus\{v\}}(G\setminus\{v\})$. Since $(J_{G}:f)=P_{T}(G)$, we have $f\not\in P_{T}(G)$ and so, $f\not\in P_{T\setminus\{v\}}(G\setminus\{v\})$. Thus, $g\in P_{T\setminus\{v\}}(G\setminus\{v\})$. Therefore, we get $(J_{G\setminus\{v\}}:f)=P_{T\setminus\{v\}}(G\setminus\{v\})$. Hence, $\mathrm{v}(J_{G\setminus \{v\}})\leq \mathrm{v}(J_{G})$.
\end{proof}

\begin{definition}[{Completion set}]\label{completionset}{\rm
Let $G$ be a simple graph. For a set $V=\{v_{1},\ldots,v_{k}\}\subseteq V(G)$, we write $G_{V}=G_{v_{1} v_{2}\cdots v_{k}}:=(\ldots ((G_{v_{1}})_{v_{2}})\ldots)_{v_{k}}$. Due to Proposition \ref{Gij}, $G_{V}$ does not depend on the ordering of the elements of $V$, and thus, the definition of $G_{V}$ is well-defined. A set $W\subseteq V(G)$ is said to be a \textit{completion set} of $G$ if $G_{W}$ is a disjoint union of complete graphs. A completion set $W$ is said to be a \textit{minimal completion set} of $G$ if $G_{U}$ is not a disjoint union of complete graphs for every $U\subsetneq W$. The minimum (respectively, maximum) cardinality among all the minimal completion sets of $G$ is denoted by $\mathrm{min\mbox{-}comp}(G)$ (respectively, $\mathrm{max\mbox{-}comp}(G)$) and we call it the \textit{minimum completion} (respectively, \textit{maximum completion}) number of $G$.
}
\end{definition}

\begin{lemma}\label{lemint}
Let $G$ be a graph. Let $T_{1},\ldots,T_{k}\in \mathscr{C}(G)\setminus\{\emptyset\}$ be some collection of cutsets of $G$. Write $I_{j}=\big<x_{i},y_{i}\mid i\in T_{j}\big>$ for each $1\leq j\leq k$. Then
$$(I_{1}+P_{\emptyset}(G))\cap\cdots\cap (I_{k}+P_{\emptyset}(G))=(I_{1}\cap\cdots\cap I_{k})+P_{\emptyset}(G).$$
\end{lemma}

\begin{proof}
It is enough to consider $G$ is connected. Note that $(I_{1}\cap\cdots\cap I_{k})+P_{\emptyset}(G)\subseteq (I_{1}+P_{\emptyset}(G))\cap\cdots\cap (I_{k}+P_{\emptyset}(G))$ is clear. We use induction on $k$ to prove the reverse inclusion. If $k=1$, then there is nothing to prove. Suppose $(I_{1}+P_{\emptyset}(G))\cap\cdots\cap (I_{k-1}+P_{\emptyset}(G))=(I_{1}\cap\cdots\cap I_{k-1})+P_{\emptyset}(G)$ holds. Let $f\in (I_{1}+P_{\emptyset}(G))\cap\cdots\cap (I_{k}+P_{\emptyset}(G))=((I_{1}\cap\cdots\cap I_{k-1})+P_{\emptyset}(G))\cap (I_{k}+P_{\emptyset}(G))$. Since $P_{\emptyset}(G)\subseteq I_{k}+P_{\emptyset}(G)$, we have
$$((I_{1}\cap\cdots\cap I_{k-1})+P_{\emptyset}(G))\cap (I_{k}+P_{\emptyset}(G))=(I_{1}\cap\cdots\cap I_{k-1})\cap (I_{k}+P_{\emptyset}(G))+P_{\emptyset}(G).$$
Then we can write $f=g+h$, where $g\in (I_{1}\cap\cdots\cap I_{k-1})\cap (I_{k}+P_{\emptyset}(G))$ and $h\in P_{\emptyset}(G)$. Note that $\{x_{i},y_{i}\mid i\in T_{k}\}\cup \{f_{pq}\mid p<q\,\,\text{and}\,\, p,q\in V(G)\setminus T_{k}\}$ is a reduced Gr\"{o}bner basis of $I_{k}+P_{\emptyset}(G)$. Therefore, we can write $g=g'+h'$ 
in the reduced form, such that $g'\in I_{k}$ and $h'\in P_{\emptyset}(G)$. Now $g\in I_{1}\cap\cdots\cap I_{k-1}$ implies each monomial term of $g'$ and $h'$ belongs to $I_{1}\cap\cdots\cap I_{k-1}$ as $I_{1}\cap\cdots\cap I_{k-1}$ is a monomial ideal. Thus, $g'\in I_{1}\cap \cdots\cap I_{k}$ and hence, $f=g'+h'+h\in (I_{1}\cap\cdots\cap I_{k})+P_{\emptyset}(G)$.
\end{proof}

\begin{theorem}\label{vphi}
Let $G$ be a simple graph. Then $\mathrm{v}_{\emptyset}(J_{G})=\mathrm{min\mbox{-}comp}(G)$. In particular, we have $\mathrm{v}(J_{G})\leq \mathrm{min\mbox{-}comp}(G)$.
\end{theorem}

\begin{proof}
Let $\mathrm{min\mbox{-}comp}(G)=k$ and $\{v_{1},\ldots,v_{k}\}$ be a minimal completion set of $G$. Then $J_{G}:x_{v_{1}}\cdots x_{v_{k}}=P_{\emptyset}(G)$ due to Proposition \ref{complete}. Therefore, $\mathrm{v}_{\emptyset}(J_{G})\leq \mathrm{min\mbox{-}comp}(G)$. For the reverse inequality, let $f$ be a homogeneous polynomial such that $(J_{G}:f)=P_{\emptyset}(G)$ and $\mathrm{v}_{\emptyset}(J_{G})=\mathrm{deg}(f)$. Then $f\in \bigcap_{T\in \mathscr{C}(G)\setminus \{\emptyset\}}P_{T}(G)$ and $f\not\in P_{\emptyset}(G)$. For every $T\in \mathscr{C}(G)$, we can write $P_{T}(G)=\big<x_{i},y_{i}\mid i\in T\big>+I_{T}$, where $I_{T}$ is a binomial edge ideal of disjoint union of some complete graphs such that $I_{T}\subseteq P_{\emptyset}(G)$. Since $f\not\in P_{\emptyset}(G)$, using Lemma \ref{lemint}, we can write $f=g+h$ such that $(0\neq) g\in I$ and $h\in P_{\emptyset}(G)$, where $I$ is the square-free monomial ideal given by
$$I=\bigcap_{T\in \mathscr{C}(G)\setminus\{\emptyset\}} \big<x_{i},y_{i}\mid i\in T\big>.$$
Thus, $\mathrm{deg}(f)\geq \mbox{min}\{\mathrm{deg}(m)\mid m\in G(I)\}$. For every $m\in G(I)$, we have 
$$m\in  \bigcap_{T\in \mathscr{C}(G)\setminus \{\emptyset\}}P_{T}(G)\,\,\text{and}\,\, m\not\in P_{\emptyset}(G).$$
Thus, $(J_{G}:m)=P_{\emptyset}(G)$ for every $m\in G(I)$. Therefore, we can choose $f$ such that $f\in G(I)$. Suppose $f=x_{i_{1}}\cdots x_{i_{r}} y_{j_{1}}\cdots y_{j_{s}}$. Then by Proposition \ref{complete}, we get $\{i_{1},\ldots,i_{r},j_{1},\ldots,j_{s}\}$ is a completion set of $G$. Thus, $\mathrm{deg}(f)=\mathrm{v}_{\emptyset}(J_{G})\geq \mathrm{min\mbox{-}comp}(G)$. Hence, $\mathrm{v}_{\emptyset}(J_{G})=\mathrm{min\mbox{-}comp}(G)$ and $\mathrm{v}(J_{G})\leq \mathrm{v}_{\emptyset}(J_{G})=\mathrm{min\mbox{-}comp}(G)$.
\end{proof}

\begin{remark}{\rm
Let $G$ be a graph with connected components $G_{1},\ldots, G_{k}$. It is easy to see from Definition \ref{completionset} that $\mathrm{min\mbox{-}comp}(G)=\mathrm{min\mbox{-}comp}(G_{1})+\cdots+\mathrm{min\mbox{-}comp}(G_{k})$. Then by Theorem \ref{vphi}, we get $\mathrm{v}_{\emptyset}(J_{G})=\mathrm{v}_{\emptyset}(J_{G_{1}})+\cdots + \mathrm{v}_{\emptyset}(J_{G_{k}})$.
}
\end{remark}

\begin{definition}{\rm
A \textit{dominating set} of a graph $G$ is a set $D\subseteq V(G)$ such that every vertex not in $D$ has a neighbour in $D$. The \textit{domination number} of $G$, denoted by $\gamma( G)$, is the cardinality of a dominating set with minimum vertices.
}
\end{definition}

\begin{corollary}\label{cordomin}
Let $G$ be a connected non-complete graph. Then $\gamma(G)\leq \mathrm{v}_{\emptyset}(G)$.
\end{corollary}

\begin{proof}
By Theorem \ref{vphi}, we have $\mathrm{v}_{\emptyset}(G)=\mathrm{min\mbox{-}comp}(G)$. Thus, it is enough to show that any minimal completion set of $G$ is a dominating set of $G$. Let $V=\{v_{1},\ldots,v_{k}\}$ be a minimal completion set of $G$. Note that $V$ is non-empty as $G$ is non-complete. Suppose there is a vertex $u\in V(G)\setminus V$ such that $u\not\in \mathcal{N}_{G}(v_{i})$ for each $1\leq i\leq k$. Then $u\not\in \mathcal{N}_{G_{V}}(v_{k})$, which gives a contradiction as $G_{V}$ is a complete graph. Thus, $V$ is a dominating set of $G$ and hence, $\gamma(G)\leq \mathrm{v}_{\emptyset}(G)$.
\end{proof}

\begin{lemma}\label{lemadd}
Let $I_{1}\subseteq R_{1}=K[x_{1},\ldots,x_{n}]$ and $I_{2}\subseteq R_{2}=K[y_{1},\ldots,y_{m}]$ be two radical ideals. Suppose $P_{1}\in \mathrm{Ass}(I_{1})$ and $I_{1}+I_{2}:=I_{1}R+I_{2}R$, where $R=R_{1}\otimes_{K} R_{2}$, is a radical ideal with $\mathrm{Ass}(I_{1}+I_{2})=\{Q_{1}+Q_{2}\mid Q_{1}\in \mathrm{Ass}(I_{1}), Q_{2}\in\mathrm{Ass}(I_{2})\}$. Then $((I_{1}+I_{2}):P_{1})=(I_{1}:P_{1})+I_{2}$.
\end{lemma}

\begin{proof}
We write $J=I_{1}+I_{2}$. Let $f\in (I_{1}:P_{1})+I_{2}$. Then $f=g+h$ for some $g\in (I_{1}:P_{1})$ and $h\in I_{2}$. Since $gP_{1}\subseteq I_{1}$ and $h\in I_{2}$, we have $fP_{1}=gP_{1}+hP_{1}\subseteq I_{1}+I_{2}=J$. Therefore, $f\in (J:P_{1})$ and $(I_{1}:P_{1})+I_{2}\subseteq (J:P_{1})$. Let $f^{\prime}\in (J:P_{1})$. Then $f^{\prime}P_{1}\subseteq J$. Since $G(P_{1})\subseteq R_{1}$ and $I_{1}$ is radical, we have $P_{1}\not\subseteq Q_{1}+ Q_{2}$ for every $Q_{1}\in \mathrm{Ass}(I_{1})\setminus\{P_{1}\}$ and $Q_{2}\in\mathrm{Ass}(I_{2})$. Therefore, $f^{\prime}\in I_{1}^{\prime}+ I_{2}$, where $I_{1}'=\bigcap_{Q_{1}\in \mathrm{Ass}(I_{1})\setminus\{P_{1}\}} Q_{1}$. Then we can write $f^{\prime}=g^{\prime}+h^{\prime}$ such that $g^{\prime}\in I_{1}^{\prime}$ and $h^{\prime}\in I_{2}$. If $g^{\prime}\in P_{1}$, then $g\in I_{1}$ and hence, $f^{\prime}\in J\subseteq ((I_{1}:P_{1})+I_{2})$. If $g^{\prime}\not \in P_{1}$, then $I_{1}:g^{\prime}=P_{1}$ as $g^{\prime}\in I_{}^{\prime}$. In this case, we get $f^{\prime}=g^{\prime}+h^{\prime}\in (I_{1}:P_{1})+I_{2}$, and hence, $(J:P_{1})=(I_{1}:P_{1})+I_{2}$.
\end{proof}

\begin{theorem}[The v-number is additive]\label{v-add}
Let $I_{1}\subseteq R_{1}=K[x_{1},\ldots,x_{n}]$ and $I_{2}\subseteq R_{2}=K[y_{1},\ldots,y_{m}]$ be two radical ideals. Suppose $I_{1}+I_{2}:=I_{1}R+I_{2}R$, where $R=R_{1}\otimes_{K} R_{2}$, is a radical ideal with $\mathrm{Ass}(I_{1}+I_{2})=\{P_{1}+P_{2}\mid P_{1}\in \mathrm{Ass}(I_{1}), P_{2}\in\mathrm{Ass}(I_{2})\}$. Then 
$$ \mathrm{v}(I_{1}+I_{2})=\mathrm{v}(I_{1})+\mathrm{v}(I_{2}).$$
\end{theorem}

\begin{proof}
Let $f_{i}\in R_{i}$ be a homogeneous polynomial and $P_{i}\in \mathrm{Ass}(I_{i})$ such that $I_{i}:f_{i}=P_{i}$ and $\mathrm{v}(I_{i})=\mathrm{deg}(f_{i})$, where $i\in\{1,2\}$. Then $f_{i}\in P$ for all $P\in \mathrm{Ass}(I_{i})\setminus \{P_{i}\}$ and $f_{i}\not\in P_{i}$ for $i=1,2$. It is easy to observe that $f_{1}f_{2}\in Q$ for all $Q\in \mathrm{Ass}(I_{1}+I_{2})\setminus\{P_{1}+P_{2}\}$ and $f_{1}f_{2}\not\in P_{1}+P_{2}$. Therefore, $((I_{1}+I_{2}):f_{1}f_{2})=P_{1}+P_{2}$ and so, $\mathrm{v}(I_{1}+I_{2})\leq \mathrm{v}(I_{1})+\mathrm{v}(I_{2})$. For the reverse inequality, we will use \cite[Theorem 10]{grv21}. Let $\mathrm{v}(I_{1}+I_{2})=\mathrm{v}_{P_{1}+P_{2}}(I_{1}+I_{2})$ for some $P_{1}\in\mathrm{Ass}(I_{1})$ and $P_{2}\in\mathrm{Ass}(I_{2})$. Let $J=I_{1}+I_{2}$, $(I_{1}:P_{1})/I_{1}=\big< g_{1}+I_{1},\ldots,g_{r}+I_{1}\big>$ and $(I_{2}:P_{2})/I_{2}=\big< h_{1}+I_{2},\ldots,h_{s}+I_{2}\big>$. Consider $\phi:R\mapsto R/J$ and we write $\phi(x)=\overline{x}$ for any $x\in R$. Then we have
\begin{align*}
&\big((I_{1}+I_{2}):(P_{1}+P_{2})\big)/J \\
=& ((I_{1}+I_{2}):P_{1})\cap ((I_{1}+I_{2}):P_{2})/J\\ 
=& ((I_{1}:P_{1})+I_{2})\cap ((I_{2}:P_{2})+I_{1})/J \quad (\text{by Lemma \ref{lemadd}})\\
=& ((I_{1}:P_{1})+I_{2})/J \cap ((I_{2}:P_{2})+I_{1})/J \quad (\text{by \cite[Lemma 3.2]{hj21}})\\
=& ((I_{1}:P_{1})+ J)/J \cap ((I_{2}:P_{2})+J)/J\\
=& ((I_{1}:P_{1})\cap (I_{2}:P_{2}) +J)/J \quad (\text{by \cite[Lemma 3.2]{hj21}})\\
=& ((I_{1}:P_{1})(I_{2}:P_{2}) +J)/J \quad (\text{by \cite[Lemma 3.1]{hntt20}})\\
=& (((I_{1}:P_{1})+J)/J) (((I_{1}:P_{1})+J)/J)\\
=& \big<\overline{g_{1}},\ldots, \overline{g_{r}}\big> \big<\overline{h_{1}},\ldots,\overline{h_{s}}\big>\\
=& \big<\overline{g_{i}h_{j}}\mid 1\leq i\leq r, 1\leq j\leq s\big>
\end{align*}
By \cite[Theorem 3.2]{grv21}, we have 
$$\mbox{v}_{P_{1}+P_{2}}(J)=\mbox{min}\{\mathrm{deg}(g_{i}h_{j})\mid 1\leq i\leq r, 1\leq j\leq s\,\, \text{and}\,\, (J:g_{i}h_{j})=P_{1}+P_{2}\}.$$
Suppose $\mathrm{v}_{P_{1}+P_{2}}(J)=\mathrm{deg}(g_{i}h_{j})=\mathrm{deg}(g_{i})+\mathrm{deg}(h_{j})$ for some $g_{i}$ and $h_{j}$. Since $J,I_{1}, I_{2}$ are radical, $J:g_{i}h_{j}=P_{1}+P_{2}$ implies $I_{1}:g_{j}=P_{1}$ and $I_{2}:h_{j}=P_{2}$. Thus, 
$$\mathrm{v}(I_{1})+\mathrm{v}(I_{2})\leq \mathrm{v}_{P_{1}}(I_{1})+\mathrm{v}_{P_{2}}(I_{2})\leq \mathrm{deg}(g_{i})+\mathrm{deg}(h_{j})=\mathrm{v}_{P_{1}+P_{2}}(I_{1}+I_{2}).$$
Since $\mathrm{v}(I_{1}+I_{2})=\mathrm{v}_{P_{1}+P_{2}}(I_{1}+I_{2})$, we have $\mathrm{v}(I_{1}+I_{2})=\mathrm{v}(I_{1})+\mathrm{v}(I_{2})$.
\end{proof}

\begin{corollary}\label{v-addbinom}
Let $G=G_{1}\sqcup G_{2}$ be a graph. Then $\mathrm{v}(J_{G})=\mathrm{v}(J_{G_{1}})+\mathrm{v}(J_{G_{2}})$.
\end{corollary}

\begin{proof}
Since binomial edge ideals are radical ideals and $\mathrm{Ass}(J_{G})=\{P_{T_{1}}(G_{1})+P_{T_{2}}(G_{2})\mid T_{1}\in\mathscr{C}(G_{1}), T_{2}\in \mathscr{C}(G_{2})\}$, we have $\mathrm{v}(J_{G})=\mathrm{v}(J_{G_{1}})+\mathrm{v}(J_{G_{2}})$ by Theorem \ref{v-add}.
\end{proof}

Now, we will establish the relation between $\mathrm{v}(J_{G})$ and $\mathrm{v}(\mathrm{in}_{<}(J_{G}))$ for certain class of graphs. For that, let us first discuss the primary decomposition of $\mathrm{v}(\mathrm{in}_{<}(J_{G}))$.
\medskip

Let $G$ be a simple graph. Let $T\in \mathscr{C}(G)$ and $G_{1},\ldots,G_{c(T)}$ be the connected components of $G\setminus T$. For $\mathbf{v}=(v_{1},\ldots,v_{c(T)})\in V(G_{1})\times\cdots\times V(G_{c(T)})$, we consider the following prime ideal:
$$P_{T}(\mathbf{v})=\big<x_{i},y_{i}\mid i\in T\big>+ \sum_{k=1}^{c(T)}\big< \{x_{i},y_{j}\mid i,j\in V(G_{k}), i<v_{k}, j>v_{k}\}\big>.$$
By \cite[Lemma 1]{s2acc}, we get
$$\mathrm{in}_{<}(P_{T}(G))=\bigcap_{\mathbf{v}\in V(G_{1})\times\cdots\times V(G_{c(T)})} P_{T}(\mathbf{v}).$$
Since $\mathrm{in}_{<}(J_{G})$ is radical, by \cite[Corollary 1.12]{cdg18}, we have 
$$\mathrm{in}_{<}(J_{G})=\bigcap_{T\in\mathscr{C}(G)}  \mathrm{in}_{<}(P_{T}(G)).$$

\begin{proposition}[{\cite[Proposition 3.4]{ki_girth}}]\label{propassinJG}
Let $G$ be a simple graph. Then we have 
$$\mathrm{Ass}(\mathrm{in}_{<}(J_{G}))=\{P_{T}(\mathbf{v})\mid T\in \mathscr{C}(G)\,\, \text{and}\,\, \mathbf{v}\in V(G_{1})\times\cdots\times V(G_{c(T)})\}.$$
\end{proposition}

\begin{lemma}\label{leminint}
Let $\mathfrak{p}_{1},\ldots,\mathfrak{p}_{k}$ be graded prime ideals of $R$. If $\mathrm{in}_{<}(\bigcap_{i=1}^{k} \mathfrak{p}_{i})$ is a square-free monomial ideal, then
$$ \mathrm{in}_{<}\bigg(\bigcap_{i=1}^{k} \mathfrak{p}_{i}\bigg)=\bigcap_{i=1}^{k} \big(\mathrm{in}_{<}(\mathfrak{p}_{i})\big).$$
\end{lemma}

\begin{proof}
We have $\mathrm{in}_{<}(\bigcap_{i=1}^{k} \mathfrak{p}_{i})\subseteq \bigcap_{i=1}^{k} (\mathrm{in}_{<}(\mathfrak{p}_{i}))$. Let $m\in \bigcap_{i=1}^{k} (\mathrm{in}_{<}(\mathfrak{p}_{i}))$ be a monomial. Then there exists $f_{i}\in \mathfrak{p}_{i}$ such that $\mathrm{in}_{<}(f_{i})=m$ for each $i\in \{1,\ldots,k\}$. Consider the polynomial $f=f_{1}\cdots f_{k}$. Note that $f\in \bigcap_{i=1}^{k} \mathfrak{p}_{i}$ and thus, $\mathrm{in}_{<}(f)=m^k\in \mathrm{in}_{<}(\bigcap_{i=1}^{k} \mathfrak{p}_{i})$. It is given that $\mathrm{in}_{<}(\bigcap_{i=1}^{k} \mathfrak{p}_{i})$ is square-free, i.e., a radical ideal. Therefore, $m\in \mathrm{in}_{<}(\bigcap_{i=1}^{k} \mathfrak{p}_{i})$ and the equality follows.
\end{proof}

\begin{theorem}\label{thmvin}
Let $G$ be a graph. If there exists $T'\in\mathscr{C}(G)$ such that $\mathrm{v}(\mathrm{in}_{<}(J_{G}))$ is attained for some $P_{T'}(\mathbf{v'})$ and $\mathrm{in}_{<}(\bigcap_{T\in \mathscr{C}(G)\setminus \{T'\}} P_{T}(G))$ is radical, then 
$$\mathrm{v}(J_G) \leq \mathrm{v}(\mathrm{in}_< (J_G)).$$
\end{theorem}

\begin{proof}
Let $G'_{1},\ldots, G'_{c(T')}$ be the connected components of $G\setminus T'$. By the given hypothesis, there exists a square-free monomial $m$ such that $(\mathrm{in}_<(J_G):m) = P_{T'}(\mathbf{v'})$ for some $\mathbf{v'}\in V(G'_{1})\times\cdots\times V(G'_{c(T')})$ and $\mathrm{v}(\mathrm{in}_{<}(J_{G}))=\mathrm{deg}(m)$. Then

\begin{align*}
& \big(\mathrm{in}_<(J_G):m \big) = P_{T'}(\mathbf{v'})\\
\implies & \bigg(\bigcap_{T \in\mathcal{C}(G)} \mathrm{in}_< (P_T(G)):m\bigg) = P_{T'}(\mathbf{v'})\\  
\implies & \bigcap_{T \in\mathcal{C}(G)} \big(\mathrm{in}_< (P_T(G)):m \big) = P_{T'}(\mathbf{v'}) \\ 
\implies  & \bigcap_{T \in\mathcal{C}(G)} \bigg( \big(\bigcap_{\mathbf{v} \in V(G_1) \times \cdots \times V(G_{\mathcal{C}(T)})} P_T (\mathbf{v})\big) : m \bigg) = P_{T'}(\mathbf{v'})\\
\implies & \bigcap_{\substack{T \in \mathcal{C}(G) \\ \mathbf{v} \in V(G_1) \times \cdots \times V(G_{\mathcal{C}(T)})}} (P_T(\mathbf{v}):m) = P_{T'}(\mathbf{v'}).
\end{align*} 
\noindent Therefore, by Proposition \ref{propassinJG}, we get $m\in P_{T}(\mathbf{v})$ for all $P_{T}(\mathbf{v})\in \mathrm{Ass}(\mathrm{in}_{<}(J_{G}))\setminus\{P_{T'}(\mathbf{v'})\}$ and $m\not\in P_{T'}(\mathbf{v'})$. This gives us that $ m\in \bigcap_{T\in \mathcal{A} } \mathrm{in}_<(P_T(G))$, where $\mathcal{A}=\mathscr{C}(G)\setminus \{T'\}$. Since $\mathrm{in}_<(\bigcap_{T \in \mathcal{A}} P_T(G))$ is radical, it follows from Lemma \ref{leminint} that
$$\bigcap_{T\in \mathcal{A} } \mathrm{in}_<(P_T(G))= \mathrm{in}_<\bigg(\bigcap_{T \in \mathcal{A}} P_T(G)\bigg).$$ 

\noindent Thus, we have $m\in \mathrm{in}_<(\bigcap_{T \in \mathcal{A}} P_T(G))$ and so, there exist $f \in \bigcap_{T\in \mathcal{A}} P_T(G) $ such that $\mathrm{in}_{<}(f) = m$. Suppose $f \in P_{T'}(G)$. Then $\mathrm{in}_{<}(f)=m \in \mathrm{in}_{<} (P_{T'}(G))$, which imply $m\in P_{T'}(\mathbf{v'})$ and this gives a contradiction as $m \not\in P_{T'}(\mathbf{v'})$. Therefore, $f \not\in P_{T'}(G)$ and we get $J_G : f = P_{T'}(G)$. Hence $\mathrm{v}(J_G) \leq \mathrm{v}(\mathrm{in}_< (J_G))$.
\end{proof}

Conca and Varbaro in \cite{cv20} introduced the notion of Knutson ideals inspired by the work of Allen Knutson \cite{knutson} on compatibly split ideals and degeneration.

\begin{definition}[{Knutson ideals}]\label{knutson}{\rm
Let $f\in R=K[x_{1},\ldots, x_{n}]$ be a polynomial such that its leading term $\mathrm{in}_{<}(f)$ is a square-free monomial for some term order $<$. Define $C_{f}$ to be the smallest set of ideals satisfying the following conditions:
\begin{enumerate}
\item $\big<f\big>\in C_{f}$;
\item If $I\in C_{f}$, then $I:J\in C_{f}$ for every ideal $J\subseteq R$;
\item If $I$ and $J$ are in $C_{f}$, then also $I+J$ and $I\cap J$ must be in $C_{f}$.
\end{enumerate} 
If $I$ is an ideal in $C_{f}$, we say that $I$ is a \textit{Knutson ideal} associated with $f$. More generally, we say that $I$ is a \textit{Knutson ideal} if $I\in C_{f}$ for some $f$.
}
\end{definition}

Matsuda introduced the notion of weakly closed graphs \cite[Definition 2.1]{matsuda18} as a generalization of closed graphs and studied the $F$-purity of binomial edge ideals of weakly closed graphs. Later, Seccia \cite[Theorem 4.1]{seccia22} proved that a graph $G$ is weakly closed if and only if $J_{G}$ is a Knutson ideal. As an application of Theorem \ref{thmvin}, we give the following corollary regarding the $\mathrm{v}$-number of binomial edge ideals of weakly closed graphs.


\begin{corollary}\label{corweakly}
Let $G$ be a weakly closed graph. Then $\mathrm{v}(J_{G})\leq \mathrm{v}(\mathrm{in}_{<}(J_{G}))$. 
\end{corollary}

\begin{proof}
Since $G$ is weakly closed, $J_{G}$ is a Knutson ideal. Then by definition of Knutson ideals, it follows that any associated prime of $J_{G}$ is Knutson. Let $\mathrm{v}(J_{G})=\mathrm{v}_{T'}(J_{G})$ for some $T'\in \mathscr{C}(G)$. Then $\bigcap_{T\in\mathscr{C}(G)\setminus\{T'\}} P_{T}(G)$ is Knutson by condition (3) of Definition \ref{knutson}. It has been proved in \cite{knutson} that the initial ideal of any Knutson ideal is radical. Thus, $G$ satisfies the hypothesis of Theorem \ref{thmvin} and hence, $\mathrm{v}(J_{G})\leq \mathrm{v}(\mathrm{in}_{<}(J_{G}))$. 
\end{proof}

\begin{figure}[H]
	\centering
	\begin{subfigure}{0.45\textwidth}
	\centering
	\begin{tikzpicture}
  [scale=1,auto=left,every node/.style={circle,scale=0.6}]
 
  \node[draw] (n1) at (0,1)  {$2$};
  \node[draw] (n2) at (-0.8660,-0.5)  {$3$};
  \node[draw] (n3) at (0.8660,-0.5) {$4$};
  \node[draw] (n4) at (0,2*1) {$1$};
  \node[draw] (n5) at (-2*0.8660,-2*0.5) {$6$};
  \node[draw] (n6) at (2*0.8660,-2*0.5)  {$5$};
 
  \foreach \from/\to in {n1/n2, n1/n3, n2/n3, n1/n4, n2/n5, n3/n6}
    \draw[] (\from) -- (\to);
    
\end{tikzpicture}
\caption{Graph $G$ with $\mathrm{v}(\mathrm{in}_{<}(J_{G}))=4$}\label{figin1}
	\end{subfigure}
	\begin{subfigure}{0.45\textwidth}
	\centering
	\begin{tikzpicture}
 [scale=1,auto=left,every node/.style={circle,scale=0.6}]
 
  \node[draw] (n1) at (0,1)  {$1$};
  \node[draw] (n2) at (-0.8660,-0.5)  {$2$};
  \node[draw] (n3) at (0.8660,-0.5) {$3$};
  \node[draw] (n4) at (0,2*1) {$4$};
  \node[draw] (n5) at (-2*0.8660,-2*0.5) {$5$};
  \node[draw] (n6) at (2*0.8660,-2*0.5)  {$6$};
 
  \foreach \from/\to in {n1/n2, n1/n3, n2/n3, n1/n4, n2/n5, n3/n6}
    \draw[] (\from) -- (\to);
   
\end{tikzpicture}
\caption{Graph $G$ with $\mathrm{v}(\mathrm{in}_{<}(J_{G}))=3$}\label{figin2}
	\end{subfigure}
	
	\caption{Graph $G$ with different labelling and different $\mathrm{v}(\mathrm{in}_{<}(J_{G}))$}\label{figin}
\end{figure}

\begin{example}\label{exmvin}{\rm
Consider the graph $G$ in \Cref{figin} with the labelling shown in (a). Then using the reduced Gr\"{o}ner basis of $J_{G}$ discussed in Remark \ref{remgrob}, we get 
$$\mathrm{in}_<(J_G) = \big< x_4y_5,x_3y_6, x_3y_4, x_2y_4, x_2y_3, x_1y_2, x_4y_3y_6, x_5y_3y_4y_6 \big>.$$
By \cite[Procedure A1]{grv21} and Macaulay2 \cite{mac2}, we get $\mathrm{v}(J_G) = 3$ and $\mathrm{v}(\mathrm{in}_<(J_G)) = 4$. Therefore, in this case, we have $\mathrm{v}(J_G) < \mathrm{v} (\mathrm{in}_<(J_G))$.\par 

Similarly, considering the same graph $G$ in \Cref{figin} with the labelling given in (b), we get $\mathrm{v}(\mathrm{in}_<(J_G)) = 3$. Thus, $\mathrm{v}(J_G) = \mathrm{v} (\mathrm{in}_<(J_G))$ in this case.\par 
Since the primary decomposition of $J_{G}$ does not depend on the labelling of $V(G)$, $\mathrm{v}(J_{G})$ remains the same for any labelling of a given graph. But, $\mathrm{v}(\mathrm{in}_{<}(J_{G}))$ may not remain same with different labelling of $V(G)$. 
}
\end{example}

We will now classify those graphs whose binomial edge ideals have $\mathrm{v}$-number $1$.
\begin{definition}{\rm
Let $G$ be a graph and $v\not\in V(G)$ be a vertex. The \textit{cone} of $v$ on $G$, denoted by $\mathrm{cone}(v,G)$, is the graph with vertex set $V(G)\cup\{v\}$ and edge set $E(G)\cup \{\{u,v\}\mid u\in V(G)\}$.
}
\end{definition}

\begin{theorem}\label{thmcone}
Let $G$ be a simple connected graph. Then $\mathrm{v}(J_{G})=1$ if and only if $G=\mathrm{cone}(v,H)$ for some non-complete graph $H$.
\end{theorem}

\begin{proof}
Suppose $\mathrm{v}(J_{G})=1$. Then there exists a homogeneous linear polynomial $f$ such that $(J_{G}:f)=P_{T}(G)$ for some $T\in \mathscr{C}(G)$. Suppose $T\neq \emptyset$. Then there exists $i\in T$ and so, $x_{i},y_{i}\in P_{T}(G)$. Therefore, $fx_{i}\in J_{G}\subseteq P_{\emptyset}(G)$. But, $P_{\emptyset}(G)$ can not contain any linear polynomial, and this gives a contradiction. Therefore, the only possibility is $T=\emptyset$, i.e., $J_{G}:f=P_{\emptyset}(G)$. Since $J_{G}=\cap_{T\in \mathscr{C}(G)} P_{T}(G)$, we have $f\in P_{T}(G)$ for all $T\in \mathscr{C}(G)$ with $T\neq \emptyset$ and $f\not\in P_{\emptyset}(G)$. Now each $P_{T}(G)$ can be written as $P_{T}(G)=\big< x_{i},y_{i}\mid i\in T\big>+I_{T}$, where $I_{T}$ is an ideal generated by degree two homogeneous binomials. Since $f$ is linear and $I_{T}$ is a homogeneous binomial ideal of degree two, we have 
$$f\in \bigcap_{T\in \mathscr{C}(G)\setminus\{\emptyset\}} \big<x_{i},y_{i}\mid i\in T\big>.$$
Now, $f$ belongs to a square-free monomial ideal and degree of $f$ is $1$ together imply there exists $v\in V(G)$ such that $v\in T$ for all $T\in \mathscr{C}(G)\setminus\{\emptyset\}$. Suppose there exists $u\in V(G)$ such that $u\not\in \mathcal{N}_{G}(v)$. Take $T\subseteq \mathcal{N}_{G}(v)$ such that there is no path from $u$ to $v$ in $G\setminus T$ and $T$ is minimal with such property. Then it is clear that $T\in \mathscr{C}(G)$ and $T\neq \emptyset$ as $G$ is connected. But, $v\not\in T$ and $T$ is non-empty give a contradiction. Hence, $\mathcal{N}_{G}[v]=V(G)$, i.e., $G=\mathrm{cone}(v,H)$, where $H$ is the induced subgraph of $G$ on $\mathcal{N}_{G}(v)$. Suppose $H$ is complete. Then $G$ is complete and in this case, $\mathrm{v}(J_{G})=0$ as $J_{G}$ is prime. Therefore, $H$ is non-complete as $\mathrm{v}(J_{G})=1$.\par

Conversely, let $G=\mathrm{cone}(v,H)$ for some non-complete graph $H$. By Proposition \ref{complete}, we get $J_{G}:x_{v}=J_{G_{v}}$. Since $G=\mathrm{cone}(v,H)$, $G_{v}$ is complete and $J_{G_{v}}=P_{\emptyset}(G)$. Thus, $\mathrm{v}(J_{G})\leq \mathrm{deg}(x_{v})=1$. Now $H$ being non-complete, $G$ is also non-complete, and so, $J_{G}$ is not a prime ideal. Therefore $\mathrm{v}(J_{G})\geq 1$, which gives $\mathrm{v}(J_{G})=1$.
\end{proof}
\medskip

\section{The $\mathrm{v}$-Number and Castelnuovo-Mumford Regularity}\label{secvreg}

In this section, we try to establish a relation between the $\mathrm{v}$-number and Castelnuovo-Mumford regularity of binomial edge ideals. For certain classes of graphs, we show that the $\mathrm{v}$-number is less than or equal to the regularity of binomial edge ideals. Our main technique is to observe the induced matchings of the hypergraphs corresponding to the initial ideals of binomial edge ideals.

\begin{definition}{\rm
A \textit{simple hypergraph} $\mathcal{H}$ is a pair $(V(\mathcal{H}),E(\mathcal{H}))$, where $V(\mathcal{H})$ is a set of finite elements, known as the \textit{vertex set} of $\mathcal{H}$ and $E(\mathcal{H})$ is a collection of subsets of $V(\mathcal{H})$ such that no two elements of $E(\mathcal{H})$ contain each other, called the \textit{edge set} of $\mathcal{H}$. Elements of $V(\mathcal{H})$ are called vertices of $\mathcal{H}$ and elements of $E(\mathcal{H})$ are called edges of $\mathcal{H}$.
}
\end{definition}

Let $\mathcal{H}$ be a simple hypergraph on the vertex set $V(\mathcal{H})=\{x_{1},\ldots,x_{n}\}$. For $A\subseteq V(\mathcal{C})$, we consider $X_{A}:=\prod_{x_{i}\in A} x_{i}$ as a square-free monomial in the polynomial ring $R=K[x_{1},\ldots,x_{n}]$ over a field $K$. The \textit{edge ideal} of the hypergraph $\mathcal{H}$, denoted by $I(\mathcal{H})$, is an ideal of $R$ defined by 
$$I(\mathcal{H})=\big<X_{e}\mid e\in E(\mathcal{H})\big>.$$
In this sense, the family of square-free monomial ideals are in one to one correspondence with the family of simple hypergraphs. For a square-free monomial ideal $I$ of $R$, we denote by $\mathcal{H}(I)$ the Corresponding simple hypergraph.

\begin{definition}{\rm
An \textit{induced matching} in a simple hypergraph $\mathcal{H}$ is a set of pairwise disjoint edges $e_{1},\ldots,e_{r}$ such that the only edges of $\mathcal{H}$ contained in $\bigcup_{i=1}^{r} e_{i}$ are $e_{1},\ldots,e_{r}$.
}
\end{definition}

\begin{proposition}[{\cite[Corollary 3.9]{mv12}}]\label{propim}
Let $\mathcal{H}$ be a simple graph and $M=\{e_{1},\ldots, e_{r}\}$ be an induced matching in $\mathcal{H}$. Then $\sum_{i=1}^{r}(\vert e_{i}\vert -1)=(\sum_{i=1}^{r}\vert e_{i}\vert)-r\leq \mathrm{reg}(R/I(\mathcal{H}))$.
\end{proposition}

\begin{lemma}\label{lemnopath}
Let $\{v_{1},\ldots,v_{k}\}$ forms a minimal completion set of a connected graph $G$ such that $v_{i}\in \mathcal{N}_{G}(v_{1})\cup \cdots\cup \mathcal{N}_{G}(v_{i-1})$ for all $2\leq i\leq k$. Then for each $2\leq i\leq k$, there exists $u_{i}\in \mathcal{N}_{G}(v_{i})$ such that $u_{i}\not\in \mathcal{N}_{G}(v_{1})\cup \cdots\cup \mathcal{N}_{G}(v_{i-1})$ and there is no path from $u_{i}$ to $v_{1}$ in $G[\{v_{1},\ldots,\widehat{v_{i}},\ldots,v_{k},u_{i}\}]$.
\end{lemma}

\begin{proof}
If $\mathcal{N}_{G}(v_{i})\subseteq \mathcal{N}_{G}(v_{1})\cup \cdots\cup \mathcal{N}_{G}(v_{i-1})$, then it is clear that $G_{v_{1}\cdots \widehat{v_{i}}\cdots v_{k}}$ is complete, and this gives a contradiction to the fact that $\{v_{1},\ldots,v_{k}\}$ is a minimal completion set of $G$. Therefore, $\mathcal{N}_{G}(v_{i})\not\subseteq \mathcal{N}_{G}(v_{1})\cup \cdots\cup \mathcal{N}_{G}(v_{i-1})$. Let $\mathcal{N}_{G}(v_{i}) \setminus (\mathcal{N}_{G}(v_{1})\cup \cdots\cup \mathcal{N}_{G}(v_{i-1}))=\{u_{i_{1}},\ldots,u_{i_{r}}\}$. Suppose for each $1\leq j\leq r$, there is a path from $u_{i_{j}}$ to $v_{1}$ in $G[\{v_{1},\ldots,\widehat{v_{i}},\ldots,v_{k},u_{i_{j}}\}]$. Then $u_{i_{j}}\in \mathcal{N}_{G}(v_{s_{j}})$ for some $s_{j}\in \{i+1,\ldots,k\}$ and there is a path from $v_{s_{j}}$ to $v_{1}$ in $G[\{v_{1},\ldots,\widehat{v_{i}},\ldots,v_{k}\}]$. Let $i^{\prime}=\mathrm{max}\{s_{1},\ldots,s_{r}\}$ ($s_i$'s need not be distinct). Then $\{u_{i_{1}},\ldots,u_{i_{r}}\}\subseteq \mathcal{N}_{G_{v_{1}\cdots\widehat{v_{i}}\cdots v_{i^{\prime}}}}(v_{i^{\prime}})$. Since there exists a path from $v_{i^{\prime}}$ to $v_{1}$ in $G[\{v_{1},\ldots,\widehat{v_{i}},\ldots,v_{k}\}]$, we have $G_{v_{1}\cdots \widehat{v_{i}}\cdots v_{k}}$ is complete, which contradicts the minimality of $\{v_{1},\ldots,v_{k}\}$. This completes the proof.
\end{proof}

\begin{theorem}\label{vregchordal}
Let $G$ be a chordal graph. Then $\mathrm{max\mbox{-}comp}(G)\leq \mathrm{reg}(S/J_{G})$. In particular, we have $\mathrm{v}(J_{G})\leq \mathrm{v}_{\emptyset}(J_{G})\leq \mathrm{max\mbox{-}comp}(G)\leq \mathrm{reg}(S/J_{G})$.
\end{theorem}

\begin{proof}
It is enough to assume $G$ is connected. Let $V=\{v_{1},\ldots, v_{k}\}$ be a minimal completion set of $G$. Since $G$ is connected, after a suitable relabelling of vertices in $V$ we get an order $v_{1},\ldots,v_{k}$ such that 
\begin{align}\label{eq1}
v_{i}\in \mathcal{N}_{G_{v_{1}\cdots v_{i-1}}}(v_{i-1})=\mathcal{N}_{G}(v_{1})\cup \cdots\cup \mathcal{N}_{G}(v_{i-1})
\end{align}
for each $i=2,\ldots,k$. From the Gr\"{o}bner basis of $J_{G}$, it is clear that $\mathrm{in}_{<}(J_{G})$ changes with the labelling of $V(G)$. Our aim is to find a labelling of $G$ for which there exists an induced matching $M$ in $\mathcal{H}=\mathcal{H}(\mathrm{in}_{<}(J_{G}))$ such that $\sum_{e\in M}(\vert e\vert -1)\geq k$. We will find such suitable labelling of $V(G)$ and corresponding induced matching of $\mathcal{H}$ in a particular algorithmic technique. Let us start.
\medskip

\noindent \textbf{Step-1:} (\textbf{Case-1A}) If there exists $u_{1}\in \mathcal{N}_{G}(v_{1})$ such that $u_{1}\not\in \mathcal{N}_{G}[v_{2}]\cup \cdots\cup \mathcal{N}_{G}[v_{k}]$, then label $v_{1}=t_{1}=1$ and $u_{1}=t_{1}+1=2$. In this case, consider the set $M=\{e_{1}\}$, where $e_{1}=\{x_{1},y_{2}\}$.\par

(\textbf{Case-1B}) If $\mathcal{N}_{G}(v_{1})\subseteq \mathcal{N}_{G}[v_{2}]\cup \cdots\cup \mathcal{N}_{G}[v_{k}]$, then take $u_{1}$ as $v_{2}$, and label $v_{1}=t_{1}=1$ and $u_{1}=v_{2}=t_{2}=t_{1}+1=2$. In this case, consider the set $M=\{e_{1}\}$, where $e_{1}=\{x_{t_{1}},y_{t_{1}+1}\}=\{x_{1},y_{2}\}$.
\medskip

\noindent \textbf{Step-2:} By Lemma \ref{lemnopath}, there exists $u_{2}\in \mathcal{N}_{G}(v_{2})$ such that $u_{2}\not\in \mathcal{N}_{G}(v_{1})$ and there is no path from $u_{2}$ to $v_{1}$ in $G[\{v_{1},\widehat{v_{2}},\ldots,v_{k},u_{2}\}]$.\par 

(\textbf{Case-2A}) If there exists $u_{2}\in\mathcal{N}_{G}(v_{2})$ such that $u_{2}\not\in \mathcal{N}_{G}[v_{1}]\cup \widehat{\mathcal{N}_{G}[v_{2}]}\cup\cdots \cup \mathcal{N}_{G}[v_{k}]$, then we choose such $u_{2}$. In this situation, we label the vertices in the following fashion:
If $v_{2}\neq u_{1}$, then label $v_{2}=t_{2}=t_{1}+2=3$ and label $u_{2}=t_{2}+1=4$. In this case, update $M$ as $M=\{e_{1},e_{2}\}$, where $e_{2}=\{x_{t_{2}}, y_{t_{2}+1}\}=\{x_{3},y_{4}\}$. If $v_{2}=u_{1}$, then we have from the previous case $u_{1}=v_{2}=t_{2}=t_{1}+1$ and label $u_{2}=t_{2}+1$. In this case, update $M$ as $M=\{e_{1},e_{2}\}$, where $e_{2}=\{x_{t_{2}},y_{t_{2}+1}\}$.\par

(\textbf{Case-2B}) Suppose $\mathcal{N}_{G}(v_{2})\subseteq \mathcal{N}_{G}[v_{1}]\cup \widehat{\mathcal{N}_{G}[v_{2}]}\cup\cdots \cup \mathcal{N}_{G}[v_{k}]$. Then there exists an $u_{2}\in \{v_{3},\ldots, v_{k}\}$ satisfying the condition of Lemma \ref{lemnopath}, otherwise $\{v_{1},\widehat{v_{2}},\ldots,v_{k}\}$ will be a minimal completion set of $G$. Let $u_{2}=v_{2+j}$ such that $j$ is the smallest. In this case, we will relabel $V$ as follows:
\begin{align*}
\text{Label}\,\,  v_{2+j}\,\, &\text{as}\,\, v_{3},\\
 v_{3}\,\, &\text{as}\,\, v_{4},\\
 \vdots\\
 v_{2+j-1}\,\, &\text{as}\,\, v_{2+j},
\end{align*}
and others will remain the same. Note that the new ordering of vertices in $V$ also satisfies the property \eqref{eq1}. So we can continue with this ordering. Now, label 
\begin{align*}
v_{2}=t_{2}=
\begin{cases}
t_{1}+1\,\, \text{if}\,\, u_{1}=v_{2}\,\, \text{in step-1}\\
t_{1}+2\,\, \text{if}\,\, u_{1}\neq v_{2}\,\, \text{in step-1}
\end{cases}
\end{align*}
\noindent and $u_{2}=v_{3}=t_{3}=t_{2}+1$. In this case, we update the set $M$ as $M=\{e_{1}, e_{2}\}$, where $e_{2}=\{x_{t_{2}},y_{t_{2}+1}\}$.
\medskip

\noindent Continue this process with the following $i$-th step:
\medskip

\noindent \textbf{Step-i:} By Lemma \ref{lemnopath}, there exists $u_{i}\in \mathcal{N}_{G}(v_{i})$ such that $u_{i}\not\in \mathcal{N}_{G}(v_{1})\cup\cdots\cup \mathcal{N}_{G}(v_{i-1})$ and there is no path from $u_{i}$ to $v_{1}$ in $G[\{v_{1},\ldots,\widehat{v_{i}},\ldots,v_{k},u_{i}\}]$.\par 

(\textbf{Case-iA}) If there exists $u_{i}\in\mathcal{N}_{G}(v_{i})$ such that $u_{i}\not\in \mathcal{N}_{G}[v_{1}]\cup \cdots\cup \widehat{\mathcal{N}_{G}[v_{i}]}\cup\cdots \cup \mathcal{N}_{G}[v_{k}]$, then we choose such $u_{i}$. In this situation, we label the vertices in the following fashion:
If $v_{i}\neq u_{i-1}$, then label $v_{i}=t_{i}=t_{i-1}+2$ and label $u_{i}=t_{i}+1$. In this case, update $M$ as $M\cup\{e_{i}\}$, where $e_{i}=\{x_{t_{i}}, y_{t_{i}+1}\}$. Now if $v_{i}=u_{i-1}$, then we have from the previous case $u_{i-1}=v_{i}=t_{i}=t_{i-1}+1$ and label $u_{i}=t_{i}+1$. In this case, update $M$ as $M\cup \{e_{i}\}$, where $e_{i}=\{x_{t_{i}},y_{t_{i}+1}\}$.\par 

(\textbf{Case-iB}) Suppose $\mathcal{N}_{G}(v_{i})\subseteq \mathcal{N}_{G}[v_{1}]\cup \cdots\cup\widehat{\mathcal{N}_{G}[v_{i}]}\cup\cdots \cup \mathcal{N}_{G}[v_{k}]$. Then there exists an $u_{i}\in \{v_{i+1},\ldots, v_{k}\}$ satisfying the condition of Lemma \ref{lemnopath}, otherwise $\{v_{1},\ldots,\widehat{v_{i}},\ldots,v_{k}\}$ will be a minimal completion set of $G$. Let $u_{i}=v_{i+j}$ such that $j$ is the smallest. In this case, we will relabel $V$ as follows:
\begin{align*}
\text{Label}\,\,  v_{i+j}\,\, &\text{as}\,\, v_{i+1},\\
 v_{i+1}\,\, &\text{as}\,\, v_{i+2},\\
 \vdots\\
 v_{i+j-1}\,\, &\text{as}\,\, v_{i+j},
\end{align*}
and others will remain the same. Note that the new ordering of vertices in $V$ also satisfies the property \eqref{eq1}. So we can continue with this ordering. Now, label 
\begin{align*}
v_{i}=t_{i}=
\begin{cases}
t_{i-1}+1\,\, \text{if}\,\, u_{i-1}=v_{i}\,\, \text{in step-(i-1)}\\
t_{i-1}+2\,\, \text{if}\,\, u_{i-1}\neq v_{i}\,\, \text{in step-(i-1)}
\end{cases}
\end{align*}
\noindent and $u_{i}=v_{i+1}=t_{i+1}=t_{i}+1$. In this case, we update the set $M$ as $M\cup \{e_{i}\}$, where $e_{i}=\{x_{t_{i}},y_{t_{i}+1}\}$.
\medskip

\noindent After completing $k$ steps, we get a set $M$ consisting of $k$ edges $e_{1},\ldots,e_{k}$ of $\mathcal{H}$, where $e_{i}=\{x_{t_{i}},y_{t_{i}+1}\}$, such that 
$$\sum_{i=1}^{k} (\vert e_{i}\vert -1)=k.$$

\noindent \textbf{Claim:} The set $M$ forms an induced matching in $\mathcal{H}$.\par
\noindent \textit{Proof of claim.} Let $\mathcal{S}=\bigcup_{e\in M} e$. Then $\mathcal{S}=\{x_{t_{1}},\ldots, x_{t_{k}}, y_{t_{1}+1},\ldots,y_{t_{k}+1}\}$. By our choice and labelling of vertices, it is clear that no two elements of $M$ intersect each other. Now we will show that the only edges of $\mathcal{H}$ contained in $\mathcal{S}$ are the edges that belong to $M$. Suppose $\{x_{t_{i}},y_{t_{j}+1}\}\in E(\mathcal{H})$ for some $i\neq j$ and $x_{t_{i}},y_{t_{j}+1}\in \mathcal{S}$. Then $t_{j}+1>t_{i}+1$ and $\{t_{i},t_{j}+1\}\in E(G)$. We have chosen $u_{j}$ such that $u_{j}\not\in \mathcal{N}_{G}(v_{1})\cup\cdots\cup \mathcal{N}_{G}(v_{j-1})$ and so, $t_{j}+1\not\in \mathcal{N}_{G}(t_{i})$ as $t_{i}< t_{j}$, which is a contradiction. Thus, $\{x_{t_{i}},y_{t_{j}+1}\}\not\in E(\mathcal{H})$ when $i\neq j$ and $x_{t_{i}},y_{t_{j}+1}\in \mathcal{S}$. Hence the only edges of $\mathcal{H}$ with cardinality two contained in $\mathcal{S}$ are the edges that belong to $M$. Suppose $e\in E(\mathcal{H})$ such that $e\not \in M$ and $e\subseteq \mathcal{S}$. Then $\vert e\vert>2$. Corresponding to $e$ there exists an admissible path $\pi:t_{i}=\alpha_{0},\ldots,\alpha_{l}$ in $G$ such that $\mathrm{supp}(u_{\pi}x_{t_{i}}y_{\alpha_{l}})\subseteq \mathcal{S}$. If one of $\alpha_{1},\ldots,\alpha_{l}$ (say $\alpha_{r}$) is $t_{p}+1$ such that $t_{p}+1\not\in\{t_{1},\ldots,t_{k}\}$, then either $\alpha_{r-1}$ or $\alpha_{r+1}$ can not belong to $\{t_{1},\ldots,t_{k}\}$ by our choice of $u_{i}$'s. Because, if $t_{p}+1\not\in \{t_{1},\ldots,t_{k}\}$, then $t_{p}+1$ is adjacent to only $t_{p}$ among $\{t_{1},\ldots,t_{k}\}$. Suppose $\alpha_{r-1}\not\in \{t_{1},\ldots,t_{k}\}$ and $\alpha_{r-1}=t_{q}+1$ for some $t_{q}\in \{t_{1},\ldots,t_{k}\}$. In this situation, we will get an induced cycle of length greater than $3$ containing the vertices $\{t_{p},t_{p}+1,t_{q}+1,t_{q}\}$ and some of $\{t_{1},\ldots, t_{k}\}$, which is a contradiction to the fact that $G$ is chordal. Similarly, we will get a contradiction if $\alpha_{r+1}\not\in \{t_{1},\ldots,t_{k}\}$. Therefore, we should have $\{\alpha_{0},\ldots,\alpha_{l}\}\subseteq \{t_{1},\ldots,t_{k}\}$. Now $\alpha_{l}=u_{j}$ for some $j$. Then $v_{j}\not\in \{\alpha_{0},\ldots,\alpha_{l}\}$ as $t_{i}<$ label of $v_{j}<$ label of $u_{j}$ as per our choice labelling. Thus, there will be a path from $u_{j}$ to $v_{1}$ in $G[\{v_{1},\ldots,\widehat{v_{j}},\ldots,v_{k}, u_{j}\}]$, which gives a contradiction due to Lemma \ref{lemnopath}. Hence, $M$ forms an induced matching in $\mathcal{H}$.
\medskip

By \cite[Corollary 2.7]{cv20}, we have $\mathrm{reg}(S/J_{G})=\mathrm{reg}(S/\mathrm{in}_{<}(J_{G}))$. Again by Proposition \ref{propim}, $k\leq \mathrm{reg}(R/\mathrm{in}_{<}(J_{G}))$ as $M$ is an induced matching in $\mathcal{H}$ with $\sum_{e\in M} (\vert e\vert-1)=k$. The completion set $V$ is chosen arbitrarily and hence, $\mathrm{max\mbox{-}comp}(G)\leq \mathrm{reg}(S/J_{G})$.
\end{proof}

\begin{theorem}\label{thmwhisker}
If a graph $G$ has a minimal completion set $\{v_{1},\ldots,v_{k}\}$ such that for each $1\leq i\leq k$ there exists $u_{i}\in \mathcal{N}_{G}(v_{i})$ such that $u_{i}\not\in \mathcal{N}_{G}[v_{1}]\sqcup \cdots\sqcup \widehat{\mathcal{N}_{G}[v_{i}]}\sqcup\cdots\sqcup \mathcal{N}_{G}[v_{k}]$, then $\mathrm{v}_{\emptyset}(J_{G})\leq k\leq \mathrm{reg}(S/J_{G})$.
\end{theorem}

\begin{proof}
Let us label the vertex $v_{i}$ as $i$ and the vertex $u_{i}$ as $k+i$ for each $1\leq i\leq k$. The remaining vertices can be labelled arbitrarily. Let $\mathcal{H}$ be the corresponding hypergraph of the $\mathrm{in}_{<}(J_{G})$ with respect to our choice of labelling. Now consider the set $M=\{e_{1},\ldots,e_{k}\}\subseteq E(\mathcal{H})$, where $e_{i}=\{x_{i},y_{k+i}\}\in E(\mathcal{H})$ for each $i=1,\ldots,k$. Looking at the Gr\"{o}bner basis of $J_{G}$, it is easy to see that $M$ forms an induced matching in $\mathcal{H}$ as $k+i\in\mathcal{N}_{G}(i)$, but $k+i\not\in \mathcal{N}_{G}[1]\sqcup \cdots\sqcup \widehat{\mathcal{N}_{G}[i]}\sqcup\cdots\sqcup \mathcal{N}_{G}[k]$. Thus, $\mathrm{reg}(S/I(\mathcal{H}))\geq \sum_{i=1}^{k}(\vert e_{i}\vert-1)=k$ by Proposition \ref{propim}. Hence, by \cite[Corollary 2.7]{cv20} and Theorem \ref{vphi}, we get $\mathrm{v}_{\emptyset}(J_{G})\leq k\leq \mathrm{reg}(S/I(\mathcal{H}))=\mathrm{reg}(S/J_{G})$.
\end{proof}

\begin{definition}{\rm
Let $G$ be a graph with $V(G)=\{v_{1},\ldots,v_{n}\}$. The \textit{whisker graph} of $G$, denoted by $W_{G}$, is the graph attaching $n$ new vertices $\{u_{1},\ldots,u_{n}\}$ to $G$ as follows:
\begin{enumerate}
\item[$\bullet$] $V(W_{G})=\{v_{1},\ldots,v_{n},u_{1},\ldots,u_{n}\},$
\item[$\bullet$] $E(W_{G})=E(G)\cup \{\{v_{i},u_{i}\}\mid i=1,\ldots,n\}.$
\end{enumerate}
}
\end{definition}

\begin{corollary}\label{corwhisker}
Let $G$ be a graph with $V(G)=[n]$. Then $\mathrm{v}_{\emptyset}(W_{G})=n\leq \mathrm{reg}(S/J_{W_{G}})$.
\end{corollary}

\begin{proof}
Note that $\{1,\ldots,n\}$ is contained in every completion set of $W_{G}$. But, $\{1,\ldots,n\}$ itself is a minimal completion set of $W_{G}$. Thus, $\{1,\ldots,n\}$ is the only minimal completion set of $W_{G}$. Hence, we get the desired result by Theorem \ref{vphi} and Theorem \ref{thmwhisker}.
\end{proof}

\begin{figure}[H]
	\centering
\begin{tikzpicture}
  [scale=1,auto=left,every node/.style={circle,scale=0.6}]
 
  \node[draw] (n1) at (0,1)  {$1$};
  \node[draw] (n2) at (-0.9510,0.3090)  {$2$};
  \node[draw] (n3) at (-0.5877,-0.8090) {$3$};
  \node[draw] (n4) at (0.5877,-0.8090) {$4$};
  \node[draw] (n5) at (0.9510,0.3090) {$5$};
  
  \node[draw] (n6) at (2*0,2*1)  {$6$};
  \node[draw] (n7) at (-2*0.9510,2*0.3090)  {$7$};
  \node[draw] (n8) at (-2*0.5877,-2*0.8090) {$8$};
  \node[draw] (n9) at (2*0.5877,-2*0.8090) {$9$};
  \node[draw] (n10) at (2*0.9510,2*0.3090) {$10$};
 
  \foreach \from/\to in {n1/n2, n1/n3, n1/n4, n1/n5, n1/n6, n2/n3, n2/n5, n2/n7, n3/n4, n3/n5, n3/n8, n4/n5, n4/n9, n5/n10}
    \draw[] (\from) -- (\to);
    
\end{tikzpicture}
\caption{Graph $G$ with $\ell (G)<\mathrm{v}_{\emptyset}(J_{G}) < \mathrm{reg}(S/J_G)$.}\label{figwhisker}
\end{figure}

\begin{example}\label{exmwhisker}{\rm
Let $G$ be the graph given in \Cref{figwhisker} and $S = \mathbb{Q}[x_1, \dots, x_{10},y_1, \dots, y_{10}]$. Using Macaulay2, we get $\mathrm{reg}(S/J_G) =6$. Also, we see that $\ell (G) =$ length of the longest induced path in $G = 4$. By Corollary \ref{corwhisker}, we get $\mathrm{v}_{\emptyset}(J_G) =5$. Thus, we have $\ell (G) < \mathrm{v}_{\emptyset}(J_G) < \mathrm{reg}(S/J_G)$. This example shows that $\mathrm{v}_{\emptyset}(J_G)$ can be a better lower bound for regularity of binomial edge ideals than the lower bound given by Matsuda and Murai \cite{mm13}.
}
\end{example}

\begin{figure}[H]
	\centering
\begin{tikzpicture}
  [scale=1,auto=left,every node/.style={circle,scale=0.6}]
 
  \node[draw] (n1) at (0,1)  {$1$};
  \node[draw] (n2) at (1,0)  {$2$};
  \node[draw] (n3) at (0,-1) {$3$};
   \node[draw] (n4) at (-1,0) {$4$};
   \node[draw] (n5) at (0,2) {$5$};
  \node[draw] (n6) at (2,0) {$6$};
  \node[draw] (n7) at (0,-2) {$7$};
  \node[draw] (n8) at (-2,0) {$8$};

  \foreach \from/\to in {n1/n2,n1/n4,n1/n5, n2/n3, n2/n6, n3/n4, n3/n7, n4/n8, n5/n6, n5/n8, n6/n7, n7/n8}
    \draw[] (\from) -- (\to);
   \end{tikzpicture}
  
  \caption{Graph $G$ with $\mathrm{v}(J_{G})=\mathrm{v}_{\emptyset}(J_G) = \mathrm{reg}(S/J_G)$}\label{figvphi=reg}
\end{figure}

\begin{example}\label{v=reg}{\rm
Let $G$ be the graph shown in \Cref{figvphi=reg} and $S = \mathbb{Q}[x_1, \dots, x_8,y_1, \dots, y_8]$. Using \cite[Procedure A1]{grv21} and Macaulay2, we get $\mathrm{v}(J_{G})=4$ and $\mathrm{reg}(S/J_G) = 4$. On the other hand, $\{1,2,3,4\}$ is a completion set of $G$. Therefore, $\mathrm{v}_{\phi}(J_{G})=4$ by Theorem \ref{vphi}. In this example, we get $\mathrm{v}(J_{G})=\mathrm{v}_{\emptyset}(J_G) = \mathrm{reg}(S/J_G)$. Hence, our given bound in Theorem \ref{thmwhisker} is sharp.
}
\end{example}

\begin{theorem}\label{reg-v}
For every $n \in \mathbb{N}$, there exist a graph $G$ such that $\mathrm{reg}(S/J_G) - \mathrm{v}(J_G) = n$. Moreover, for every $n\in \mathbb{N}$, there exists a graph $G$ such that $\mathrm{reg}(S/J_G) - \mathrm{v}_{\emptyset}(J_G) = n$.

\end{theorem}
\begin{proof}
For $n=0$, the result follows from Example \ref{v=reg}. Let $n\in \mathbb{N}^+$. Consider the graph $H = P_{n+2}$, a path graph on $n+2$ vertices. Then by \cite[Corollary 7.35]{hho}, we have $\mathrm{reg}(S/J_H) = n+1$. Now, let $G = \mathrm{cone}(v,H)$, where $v \notin V(H)$. Then $\mathrm{v}(J_G) =\mathrm{v}_{\emptyset}(G)= 1$ by Theorem \ref{thmcone}. Also, using \cite[Theorem 2.1]{ks18}, we get $\mathrm{reg}(S/J_G) = n+1$. Thus, we get $\mathrm{reg}(S/J_G) - \mathrm{v}(J_G) = (n + 1) - 1 = n$. In this case, $\mathrm{v}(J_{G})=\mathrm{v}_{\emptyset}(J_{G})$ and the further hypothesis follows.
\end{proof}
\medskip

\section{Some Open Problems on $\mathrm{v}(J_{G})$}\label{secprob}

In Section \ref{v-binomprop}, we discuss some properties of $\mathrm{v}$-number of binomial edge ideals and give a combinatorial bound. In \cite{v}, the authors managed to give a combinatorial description of $\mathrm{v}$-number for edge ideals of graphs. We ask the following question on the combinatorial aspects of the $\mathrm{v}$-number of binomial edge ideals.

\begin{question}{\rm
Let $G$ be a simple graph. Can we find some homogeneous polynomial $f$ just using the combinatorics of the graph $G$ such that $\mathrm{v}(J_{G})=\mathrm{deg}(f)$? Equivalently, does there exists any graph invariant of $G$ which is equal to $\mathrm{v}(J_{G})$?
}
\end{question}

In Theorem \ref{thmvin}, we prove that $\mathrm{v}(J_{G})\leq \mathrm{v}(\mathrm{in}_{<}(J_{G}))$ for some classes of binomial edge ideals and as an application, we get in Corollary \ref{corweakly} that $\mathrm{v}(J_{G})\leq \mathrm{v}(\mathrm{in}_{<}(J_{G}))$ hold for weakly closed graphs. Also, we see in Example \ref{exmvin} that $\mathrm{v}(\mathrm{in}_{<}(J_{G}))$ depends on the labelling of vertices and $\mathrm{v}(J_{G})$ can be strictly less than $\mathrm{v}(\mathrm{in}_{<}(J_{G}))$. With the virtue of these results and our computation, we put the following question.

\begin{question}{\rm
Is it true that $\mathrm{v}(J_{G})\leq \mathrm{v}(\mathrm{in}_{<}(J_{G}))$ for all graph $G$ with all possible labelling of $V(G)$? If not, then can we say that for a graph $G$, there exists a labelling of $V(G)$ for which $\mathrm{v}(J_{G})\leq \mathrm{v}(\mathrm{in}_{<}(J_{G}))$ hold?
}
\end{question}

In Section \ref{secvreg}, we try to relate the $\mathrm{v}$-number with (Castelnuovo-Mumford) regularity of binomial edge ideals. In Theorem \ref{vregchordal} and Theorem \ref{thmwhisker}, we show that $\mathrm{v}_{\emptyset}(J_{G})\leq \mathrm{reg}(S/J_{G})$ for some large classes of graphs including chordal and whisker graphs. Using \cite[Procedure A1]{grv21} and Macaulay2 \cite{mac2}, we investigate many graphs from several classes and witness that $\mathrm{v}_{\emptyset}(J_{G})\leq \mathrm{reg}(S/J_{G})$ hold for all of those graphs. Our strong intuition forces us to give the following conjecture.

\begin{conjecture}\label{conjvnum}
Let $G$ be a simple graph. Then $\mathrm{v}_{\emptyset}(J_{G})\leq \mathrm{reg}(S/J_{G})$. In particular, we have $\mathrm{v}(J_{G})\leq \mathrm{reg}(S/J_{G})$.
\end{conjecture}

\bibliographystyle{amsalpha}

\end{document}